\documentclass[12pt,a4paper]{amsart}
\usepackage[utf8]{inputenc}
\usepackage[english]{babel}
\usepackage{amsmath,amssymb,amsfonts}
\usepackage[a4paper,margin=25mm]{geometry}
\usepackage{cite}

\newtheorem{lemma}{Lemma}
\newtheorem{theorem}{Theorem}
\newtheorem{assertion}{Proposition}
\numberwithin{equation}{section}

\title[Genericity of singularities on compact sets]{Genericity of singularities of solutions to quasilinear systems on compact sets}
\author[A.M. Shavlukov]{A.M. Shavlukov}
\address{Institute of Mathematics with Computing Centre,
a subdivision of the Ufa Federal Research Centre
of the Russian Academy of Sciences,
112 Chernyshevsky St., Ufa 450008, Russia}
\email{aza3727@yandex.ru}
\date{}
\subjclass[2020]{58K05; 35L40, 35J46}
\keywords{gas dynamics, nonlinear geometrical optics,
gradient catastrophe, singularity theory of differentiable mappings}

\begin{document}
\begin{abstract}
We study the genericity of singularities of types
$A_2,A_3,D_4^\pm$ in solutions of the equations of gas dynamics and nonlinear
geometrical optics. The sets of solutions satisfying the corresponding
transversality conditions are proved to be open and dense in spaces of
smooth and real-analytic solutions equipped with the Whitney topology on compact sets.
\end{abstract}
\maketitle

\section{Introduction}

The study of singularities of solutions of partial differential equations is of
interest both in theory and in applications to gas dynamics, hydrodynamics,
and optics. Beginning with the work of A.~Kh.~Rakhimov \cite{rakh92,rakh93},
singularity theory of differentiable mappings has been used to study generic
singularities of solutions to quasilinear systems of two equations for two
unknown functions
\cite{dub06,dgk09,ko17,ko20,ko13,bogaevsky,kud95,kud01,gar10,suldis,sul12,chap24,proval,ss22,ss24,mss24,shav,accepted}.

The papers \cite{proval,ss22,ss24,mss24,shav,accepted} describe singularities
of types $A_2,A_3,D_4^\pm$ arising in solutions of the one-dimensional gas
dynamics equations and their elliptic counterpart, the equations of nonlinear
geometrical optics. The singularities are considered up to $RL$-equivalence.
The aim of this paper is to establish the genericity of these singularities
of solutions in the class under consideration. We give the corresponding
transversality conditions and prove that the sets of solutions satisfying
them are open and dense in the Whitney topology. These open dense sets also
contain solutions with no singularities of the listed types on the chosen compact set.
The results concern restrictions of solutions to a sufficiently small compact
neighborhood of the hodograph image of a gradient catastrophe point.

As in \cite{proval,ss22,ss24,mss24,shav,accepted}, our analysis of these
singularities uses the relations among the Taylor coefficients of solutions
that follow from the equation and the vanishing of the Jacobian of the
hodograph transformation. No additional equality constraints are imposed on
the coefficients. This does not rule out solutions with singularities of
types $A_k$, $k>3$, or solutions without singularities of the types listed
above. For example, the identically zero solution of the corresponding
linear equation has no singularities of these types. This paper concerns
the genericity of precisely the listed singularities.
The transversality conditions are expressed in terms of these coefficients
and compared with the conditions in the cited papers.
The topology is defined on restrictions of solutions to a compact set.

Lemmas 1--3 establish transversality conditions for mappings. These conditions
are used in Theorems 1--3 to prove genericity of singularities of solutions.

The author thanks V.~D.~Sedykh for helpful discussions of Lemma~1.

\section{Hyperbolic system: $\rho_*>0$}

\subsection{}
Consider the space $C^\infty(U,\mathbb R^2)$ of smooth mappings
$f:(r,l)\mapsto(x,t)$, where $U$ is an open subset of the plane.
Equip this space with the Whitney $C^\infty$ topology
(see, for example, \cite{sedykh}).

\begin{lemma}\label{lem1}
Let $L$ be the closed stratified subset of the space of 1-jets defined by
\[
L=\{t_rt_l=0\}.
\]
The mappings $f$ for which $j^1f$ is transverse to every stratum of $L$ form
an open dense subset of $C^\infty(U,\mathbb R^2)$.
For these mappings, all critical points of $t$ are nondegenerate, and
$t_r,t_l$ have only regular zeros.
\end{lemma}
\begin{proof}
The stratification consists of five strata:
\[
\begin{array}{ll}
L_0=\{t_r=t_l=0\},&\operatorname{codim}L_0=2,\\
L_1^\pm=\{t_r=0,\ \pm t_l>0\},&\operatorname{codim}L_1^\pm=1,\\
L_2^\pm=\{t_l=0,\ \pm t_r>0\},&\operatorname{codim}L_2^\pm=1.
\end{array}
\]
Whitney's first condition, condition~(a), is satisfied: the tangent space
to $L_0$ is contained in the limiting tangent spaces of the adjacent strata.
At each point under consideration, the transversality conditions are
\begin{equation}\label{trans1}
\begin{aligned}
t_r=t_l=0&\ \Longrightarrow\ t_{rr}t_{ll}-t_{rl}^2\ne0,\\
t_r=0,\ t_l\ne0&\ \Longrightarrow\ (t_{rr},t_{rl})\ne(0,0),\\
t_l=0,\ t_r\ne0&\ \Longrightarrow\ (t_{rl},t_{ll})\ne(0,0).
\end{aligned}
\end{equation}
All derivatives here are evaluated at the point under consideration.
Density follows from Thom's jet transversality theorem for a finite family
of submanifolds \cite[Ch.~3, \S\,2, Theorem~2.9]{hirsch}.
Openness follows from the closedness of $L$, Whitney condition~(a), and
Trotman's theorem \cite[Theorem and Remark~(i)]{trotman}.
The mapping $f\mapsto j^1f$ is continuous in the corresponding Whitney topologies.
\end{proof}
The analogous statement holds with $t$ replaced by $x$ in the definition
of $L$ and in the transversality conditions. The coordinates $(r,l)$ and
$(x,t)$ in the lemma are fixed; its conclusion need not be preserved under
a change of the coordinates $(x,t)$.
When the lemma is applied to solutions, the relation between these conditions
and singularity types up to $RL$-equivalence is established using the
formulas for the derivatives of the hodograph transformation.

\subsection{}
Consider the system of equations of one-dimensional isentropic gas dynamics
(see, for example, \cite{ry})
\begin{equation}\label{gas}
\begin{cases}
u_t+uu_x+\alpha(\rho)\rho_x=0,\\
\rho_t+(\rho u)_x=0.
\end{cases}
\end{equation}
Here $x$ is the spatial coordinate, $t$ is time, $u(x,t)$ is the flow velocity,
$\rho(x,t)>0$ is the density, and $p(\rho)$ is the pressure;
$\alpha(\rho)=p'(\rho)/\rho$ is a positive analytic function.

Let $(x_*,t_*)$ be the gradient catastrophe point under consideration, and
let $(u_*,\rho_*)$, $\rho_*>0$, be the finite limits of the solution
as $(x,t)$ approaches $(x_*,t_*)$ within the chosen sector.
The Riemann invariants
\[
r=u+\int_{\rho_*}^{\rho}\frac{c(s)}s\,ds,\qquad
l=u-\int_{\rho_*}^{\rho}\frac{c(s)}s\,ds,
\qquad c(s)=\sqrt{s\alpha(s)}
\]
transform \cite[Ch.~2, \S\,2]{ry} system~\eqref{gas} into the diagonal system
\begin{equation}\label{riemann}
\begin{cases}
r_t+\left(\dfrac{r+l}{2}+c\right)r_x=0,\\
l_t+\left(\dfrac{r+l}{2}-c\right)l_x=0,
\end{cases}
\end{equation}
where $c(\rho)=\sqrt{\rho\alpha(\rho)}$ and $\rho$ is determined by $r-l$.
Using the formulas for the derivatives of the hodograph transformation
(for example, \cite[Ch.~1, \S\,4]{ry})
\[
\begin{gathered}
r_x=Jt_l,\quad r_t=-Jx_l,\quad l_x=-Jt_r,\quad l_t=Jx_r,\\
J=r_xl_t-r_tl_x,\qquad j=x_rt_l-x_lt_r,\qquad J=j^{-1}
\end{gathered}
\]
we transform the nonlinear system~\eqref{riemann} into the linear system
\begin{equation}\label{hodhyp}
\begin{cases}
x_l=\left(\dfrac{r+l}{2}+c\right)t_l,\\
x_r=\left(\dfrac{r+l}{2}-c\right)t_r.
\end{cases}
\end{equation}
Moreover,
\[
j=-2ct_rt_l.
\]
Thus, at $(r_*,l_*)$, the condition $j=0$ is equivalent to $j^1f$ belonging
to the set $L$ of Lemma~\ref{lem1}.

Using Shvartsburg's transformation \cite{shvartsburg}, introduce
the potential $B=B(u,\rho)$:
\begin{equation}\label{schwarz}
x=uB_u-B-\rho B_\rho,\qquad t=B_u.
\end{equation}
As in \cite{suldis,ss22,ss24,shav}, consider the function
\[
F(u,\rho;x,t)=\rho\bigl(ut-x-B(u,\rho)\bigr).
\]
For $\rho>0$, its critical points with respect to $u,\rho$ are given by
\eqref{schwarz}. In the variables $r,l$, we have $t=B_r+B_l$.
Substituting \eqref{schwarz} into \eqref{hodhyp} gives the equation
\begin{equation}\label{bhyp}
B_{rl}=\kappa(B_r-B_l),\qquad
\kappa=\frac{3+\rho\alpha'/\alpha}{8c}.
\end{equation}
We first assume that
\[
3\alpha(\rho_*)+\rho_*\alpha'(\rho_*)\ne0.
\]
Fix a sufficiently small open neighborhood of $(r_*,l_*)$ in which
$\kappa\ne0$. The choice of this neighborhood is independent of the
original solution $B$. Consider smooth solutions of \eqref{bhyp} in this
neighborhood and choose a compact neighborhood $K$ of $(r_*,l_*)$ within it.
On restrictions of solutions to $K$, the Whitney topology coincides with
the topology of uniform convergence of all derivatives on $K$.

\begin{theorem}\label{th1}
The set of smooth solutions of \eqref{bhyp} satisfying the following condition
at every point of $K$ is open and dense in the stated topology:
\begin{equation}\label{good1}
t_r=t_l=0\quad\Longrightarrow\quad t_{rr}t_{ll}\ne0.
\end{equation}
In particular, the transversality conditions~\eqref{trans1} hold for
$t=B_r+B_l$.
\end{theorem}
\begin{proof}
We first show that \eqref{good1} is equivalent to transversality to the
stratum $L_0$. Since $\kappa(r,l)$ depends only on $r-l$, we have
$\kappa_r+\kappa_l=0$. Differentiating \eqref{bhyp}, we obtain
\[
t_{rl}=(\partial_r+\partial_l)B_{rl}
=\kappa(B_{rr}-B_{ll})=\kappa(t_r-t_l).
\]
At every point of $K$ where $t_r=t_l=0$, we have $t_{rl}=0$.
Hence the matrix of derivatives of $(t_r,t_l)$ with respect to $(r,l)$ is
\[
\begin{pmatrix}t_{rr}&0\\0&t_{ll}\end{pmatrix}.
\]
Its nonsingularity is equivalent to $t_{rr}t_{ll}\ne0$ at that point.
If $t_r=0$, $t_l\ne0$, then $t_{rl}=-\kappa t_l\ne0$;
if $t_l=0$, $t_r\ne0$, then $t_{rl}=\kappa t_r\ne0$.
Thus the transversality conditions for the strata of codimension one hold.
Consequently, \eqref{good1} is equivalent to transversality to $L_0$.
It remains to prove that the solutions satisfying this condition form
an open dense set.

Set $\Delta r=r-r_*$, $\Delta l=l-l_*$.
At $(r_*,l_*)$, where $t_r=t_l=0$, the cubic part of the expansion of $F$
with $x=x_*,t=t_*$ fixed is
\[
A_+(\Delta r)^3+A_-(\Delta l)^3,\qquad A_+=-\frac{\rho_*}{6}t_{rr}(r_*,l_*),\qquad
A_-=-\frac{\rho_*}{6}t_{ll}(r_*,l_*).
\]
Thus \eqref{good1} is equivalent to $A_+A_-\ne0$, the nondegeneracy condition
on the cubic part used in \cite{ss22} for a section of the hyperbolic
umbilic catastrophe $D_4^+$.

When $t_r=0$, $t_l\ne0$, the coefficient conditions in \cite{ss22} take
the form $t_{rr}\ne0$ for a fold and $t_{rr}=0$, $t_{rrr}\ne0$ for a cusp.
When $t_r=0$, $t_l\ne0$, $t_{rr}=0$, we assume $t_{rrr}\ne0$, as in
\cite{ss22}. This condition on the Taylor coefficients is part of the
problem formulation. The case $t_l=0$, $t_r\ne0$ is treated similarly.

To prove density, fix an arbitrary smooth solution $B$ of \eqref{bhyp}.
The existence of smooth solutions with admissible Taylor expansions was
established in \cite[Lemma~1]{ss22}.
Construct four solutions $B_i$ with the Goursat data
\[
\begin{array}{c|rrrr}
i&1&2&3&4\\ \hline
B_i(r,l_*)&(\Delta r)^2/2&0&(\Delta r)^3/6&0\\
B_i(r_*,l)&0&(\Delta l)^2/2&0&(\Delta l)^3/6
\end{array}.
\]
Smooth solutions of this Goursat problem in a sufficiently small rectangle
are constructed by successive approximations
\cite[Ch.~V, \S\,6, items~1--3]{courant}.

Set $B_\lambda=B+\sum_{i=1}^4\lambda_iB_i$ and
$t_\lambda=(B_\lambda)_r+(B_\lambda)_l$.
The matrix
\[
\frac{\partial\bigl((t_\lambda)_r,(t_\lambda)_l,
(t_\lambda)_{rr},(t_\lambda)_{ll}\bigr)}
{\partial(\lambda_1,\lambda_2,\lambda_3,\lambda_4)}
\]
is triangular with ones on the diagonal at $(r_*,l_*)$.
This follows from the Goursat data and equation~\eqref{bhyp}.
The matrix is independent of the original solution $B$ and the family
parameters. By continuity, it is nonsingular in a sufficiently small
neighborhood of this point. In choosing the initial open neighborhood,
we require all $B_i$ to be defined there and this matrix to be nonsingular.
Such a choice is possible and is independent of the original solution $B$.

Each of the following systems consists of three independent equations
in the six variables $(r,l,\lambda)$:
\[
(t_\lambda)_r=(t_\lambda)_l=(t_\lambda)_{rr}=0,\qquad
(t_\lambda)_r=(t_\lambda)_l=(t_\lambda)_{ll}=0.
\]
Here $(r,l)$ ranges over an open neighborhood of $K$.
We apply the parametric transversality theorem
\cite[Ch.~3, \S\,2, Theorem~2.7]{hirsch} to the families of mappings
defined by the left-hand sides of these two systems. The nonzero minors
found above ensure the required transversality of the joint mappings.

For each degeneracy, consider the set of points and parameters at which
it occurs. By the regular value theorem \cite{hirsch}, this set is a
three-dimensional submanifold or is empty.
By Sard's theorem, the image of its projection onto the four-dimensional
parameter space has measure zero \cite[Ch.~3, \S\,1, Theorem~1.3]{hirsch}.
Outside the union of these two images, neither degeneracy occurs at any
point of $K$. Such parameters can be chosen arbitrarily close to zero.
As $\lambda\to0$, the solutions $B_\lambda$ converge to $B$ together with
all derivatives on $K$. Since the original solution $B$ is arbitrary, every
neighborhood of it contains a solution satisfying \eqref{good1}.
This proves density.

We now prove openness. Condition~\eqref{good1} excludes simultaneous
vanishing of $t_r,t_l$ with either $t_{rr}$ or $t_{ll}$. It is therefore
equivalent to positivity on $K$ of the two sums
\[
t_r^2+t_l^2+t_{rr}^2,\qquad t_r^2+t_l^2+t_{ll}^2.
\]
Denote either sum by $G$; its minimum $m=\min_K G$ is positive.
For a nearby solution $\widetilde B$, set
$\widetilde t=\widetilde B_r+\widetilde B_l$ and obtain $\widetilde G$
by replacing $t$ with $\widetilde t$ in the expression for $G$.
Suppose that the absolute values of the three derivatives occurring in
$G$ do not exceed $C$ on $K$. Assume that each corresponding derivative
of $\widetilde t$ differs from the same derivative of $t$ by less than
$\delta>0$. Then
\[
\sup_K|\widetilde G-G|\le3\delta(2C+\delta).
\]
Choosing $\delta$ sufficiently small gives
$\sup_K|\widetilde G-G|<m/2$, and hence $\widetilde G>m/2>0$ on $K$.
These bounds hold if the derivatives of $\widetilde B$ up to order three
are uniformly close enough to those of $B$ on $K$. Applying the argument to both sums shows that
\eqref{good1} also holds for all sufficiently close solutions.
\end{proof}

\textbf{Remark 1.}
The compactness of $K$ is essential to the proof.
The stated conditions are equivalent to positivity of certain sums of
squares of derivatives. On a compact set, each such sum attains a positive
minimum, so the required inequalities hold for all sufficiently close solutions.

In an open neighborhood, these sums may tend to zero as the boundary is
approached. In the Whitney topology, closeness of solutions is specified
by inequalities bounding the absolute differences of their derivatives.
The upper bounds in these inequalities may depend on the point and tend
to zero at the boundary of the domain. Making the constant parameters of
the constructed family of solutions smaller does not necessarily ensure
these inequalities throughout the open neighborhood.
On an open domain, the corresponding set of solutions is also open in
the Whitney topology. However, the density proof given above does not
carry over directly. This is a limitation of the proof, not an assertion
that density fails. The same difficulty arises in the elliptic case and in
the case $\rho_*=0$ considered below.

\textbf{Remark 2.}
The Chaplygin gas violates the genuine nonlinearity condition.
The Bechert--Stanyukovich gas satisfies this condition but violates a
stronger condition \cite[\S\,1.3]{shav}; see also \cite{chap24}.
Both cases are used to approximate isentropic gas flow.
Theorem~\ref{th1} applies to the Bechert--Stanyukovich gas.

\textbf{Remark 3.}
For the Chaplygin gas, $p(\rho)=p_0-M^2/\rho$ and
$\alpha(\rho)=M^2\rho^{-3}$, where $M>0$ is a constant.
In this case $\kappa=0$, so \eqref{bhyp} reduces to $B_{rl}=0$.

\begin{assertion}\label{chap}
The set of smooth solutions of $B_{rl}=0$ satisfying the following conditions
at every point of $K$ is open and dense in the topology on $K$:
\begin{equation}\label{goodchap}
B_{rr}=0\ \Longrightarrow\ B_{rrr}\ne0,\qquad
B_{ll}=0\ \Longrightarrow\ B_{lll}\ne0.
\end{equation}
For these solutions, the transversality conditions~\eqref{trans1} hold
for $t=B_r+B_l$.
\end{assertion}
\begin{proof}
In a rectangular neighborhood containing $K$, the solutions have the form
$B=f(r)+g(l)$. Hence
\[
B_{rr}=f'',\qquad B_{ll}=g'',\qquad B_{rrr}=f''',\qquad B_{lll}=g'''.
\]
To prove density, consider the exact solutions
\[
B_{\varepsilon,\eta}=B+\frac{\varepsilon r^2}{2}+\frac{\eta l^2}{2}.
\]
For these solutions, $(B_{\varepsilon,\eta})_{rr}=f''(r)+\varepsilon$,
$(B_{\varepsilon,\eta})_{ll}=g''(l)+\eta$, while
$(B_{\varepsilon,\eta})_{rrr}=f'''(r)$ and
$(B_{\varepsilon,\eta})_{lll}=g'''(l)$.
By Sard's theorem, arbitrarily small $\varepsilon,\eta$ can be chosen so
that $-\varepsilon,-\eta$ are regular values of $f'',g''$ on open intervals
containing the projections of $K$. Then \eqref{goodchap} holds at every
point of $K$, and $B_{\varepsilon,\eta}\to B$ in the chosen topology.

We now prove openness. Conditions~\eqref{goodchap} are equivalent to
positivity of $B_{rr}^2+B_{rrr}^2$ and $B_{ll}^2+B_{lll}^2$.
Both sums have positive minima on the compact set $K$.
As in the proof of Theorem~\ref{th1}, this ensures that the conditions
hold for all sufficiently close solutions.

Since $B_{rl}=0$ and $t=B_r+B_l$, we have
\[
t_r=B_{rr},\quad t_l=B_{ll},\quad t_{rr}=B_{rrr},\quad
t_{ll}=B_{lll},\quad t_{rl}=0.
\]
At a point where $t_r=t_l=0$, we have $B_{rr}=B_{ll}=0$.
Thus \eqref{goodchap} implies both $B_{rrr}\ne0$ and $B_{lll}\ne0$.
Consequently,
\[
t_{rr}t_{ll}-t_{rl}^{\,2}=B_{rrr}B_{lll}\ne0,
\]
which, by Lemma~\ref{lem1}, ensures transversality to the stratum $L_0$
of codimension two.
If $t_r=0$, $t_l\ne0$, then $t_{rr}\ne0$;
if $t_l=0$, $t_r\ne0$, then $t_{ll}\ne0$.
Thus the transversality conditions for the strata of codimension one hold.
\end{proof}

\section{Elliptic system: $\rho_*>0$}

\subsection{}
Let $U$ be an open neighborhood of $(u_*,\rho_*)$ in the $(u,\rho)$ plane.
Equip the space of real-analytic mappings $f:U\to\mathbb R^2$,
$f:(u,\rho)\mapsto(x,t)$, with the Whitney topology induced from the
space $C^\infty(U,\mathbb R^2)$ of smooth mappings.

\begin{lemma}\label{lem2}
Let $L$ be the closed stratified subset of the space of 1-jets defined by
\[
L=\{t_u=t_\rho=0\}.
\]
The mappings $f$ for which $j^1f$ is transverse to $L$ form an open dense
subset of the space of analytic mappings. For these mappings, all critical
points of $t$ are nondegenerate.
\end{lemma}
\begin{proof}
The set $L$ has a single stratum of codimension two.
Whitney condition~(a) is satisfied. At each point under consideration,
the transversality condition is
\begin{equation}\label{trans2}
t_u=t_\rho=0\quad\Longrightarrow\quad
\det\begin{pmatrix}t_{uu}&t_{u\rho}\\t_{\rho u}&t_{\rho\rho}\end{pmatrix}\ne0.
\end{equation}
In the smooth class, openness and density follow from Thom's theorem.
Choose a smooth transverse mapping sufficiently close to the given analytic
mapping. By openness of transversality, a sufficiently close analytic approximation
is also transverse and remains sufficiently close to the original mapping.
The existence of such an analytic approximation in the Whitney topology
follows from Proposition~8 of \cite{grauert}.
\end{proof}

\subsection{}
Consider the system of equations of nonlinear geometrical optics
\[
\begin{cases}
u_t+uu_x-\alpha(\rho)\rho_x=0,\\
\rho_t+(\rho u)_x=0.
\end{cases}
\]
Here $\alpha(\rho)>0$ is analytic in a neighborhood of $\rho_*>0$.
The point $(x_*,t_*)$, the limiting values $(u_*,\rho_*)$, and the sector
are chosen as in Section~2.2.
The formulas for the derivatives of the hodograph transformation are
\[
\begin{gathered}
u_x=Jt_\rho,\quad u_t=-Jx_\rho,\quad
\rho_x=-Jt_u,\quad \rho_t=Jx_u,\\
J=u_x\rho_t-u_t\rho_x,\qquad j=x_ut_\rho-x_\rho t_u,\qquad J=j^{-1}.
\end{gathered}
\]
They give the linear system
\[
\begin{cases}
x_\rho=ut_\rho+\alpha t_u,\\
x_u=ut_u-\rho t_\rho,
\end{cases}
\]
and
\[
j=-\rho t_\rho^2-\alpha t_u^2.
\]
Thus $j(u_*,\rho_*)=0$ is equivalent to simultaneous vanishing of
$t_u(u_*,\rho_*)$ and $t_\rho(u_*,\rho_*)$.
Transformation~\eqref{schwarz} leads to the equation
\begin{equation}\label{bell}
\rho B_{\rho\rho}+2B_\rho=-\alpha(\rho)B_{uu}.
\end{equation}
The existence of analytic solutions of this equation was
established in \cite{ss24}.
First fix a sufficiently small open neighborhood of $(u_*,\rho_*)$.
Its choice is independent of the original solution $B$.
Consider analytic solutions there and choose a compact neighborhood $K$
of $(u_*,\rho_*)$ lying in the half-plane $\rho>0$.
The topology on restrictions of solutions to $K$ is defined by uniform
convergence of all derivatives.

\begin{theorem}\label{th2}
The set of analytic solutions of \eqref{bell} satisfying the following condition
at every point of $K$ is open and dense:
\begin{equation}\label{good2}
t_u=t_\rho=0
\quad\Longrightarrow\quad
(t_{uu},t_{u\rho})\ne(0,0).
\end{equation}
Wherever $t_u=t_\rho=0$ on $K$, the transversality condition~\eqref{trans2} holds for $t=B_u$.
\end{theorem}
Let $b_{ij}$ denote the Taylor coefficients of $B$ at the point under
consideration. In the variables $u,\rho$, they are
\[
b_{ij}=\frac{1}{i!\,j!}
\frac{\partial^{i+j}B}{\partial u^i\partial\rho^j}(u_*,\rho_*).
\]
In particular, $t_{uu}=6b_{30}$ and $t_{u\rho}=2b_{21}$ at that point.
Condition~\eqref{good2} excludes simultaneous vanishing of $b_{30}$ and $b_{21}$.
The relation between these coefficients and a section of the elliptic
umbilic catastrophe $D_4^-$ was established in \cite{ss24}.
\begin{proof}
To prove density, fix an arbitrary analytic solution $B$ of \eqref{bell}.
Set $\Delta u=u-u_*$ and use the Cauchy--Kovalevskaya theorem
\cite[Ch.~I, \S\,7]{courant} to construct four analytic solutions $B_i$
with the data
\[
\begin{array}{c|rrrr}
i&1&2&3&4\\ \hline
B_i(u,\rho_*)&(\Delta u)^2/2&0&(\Delta u)^3/6&0\\
(B_i)_\rho(u,\rho_*)&0&\Delta u&0&(\Delta u)^2/2
\end{array}.
\]
For $B_\lambda=B+\sum_{i=1}^4\lambda_iB_i$, the matrix
\[
\frac{\partial\bigl((B_\lambda)_{uu},(B_\lambda)_{u\rho},
(B_\lambda)_{uuu},(B_\lambda)_{uu\rho}\bigr)}
{\partial(\lambda_1,\lambda_2,\lambda_3,\lambda_4)}
\]
is the identity matrix at $(u_*,\rho_*)$. It is independent of the original solution $B$ and the family parameters. By continuity, it is nonsingular
in a sufficiently small neighborhood of this point.
In choosing the initial open neighborhood, we require all $B_i$ to be
defined there and this matrix to be nonsingular. Such a choice is possible
and is independent of the original solution $B$.
By the regular value theorem \cite{hirsch}, simultaneous vanishing of
these four derivatives in the space of variables $(u,\rho,\lambda)$
defines a two-dimensional submanifold or the empty set.
By Sard's theorem, the image of its projection onto the four-dimensional
parameter space has measure zero.
Hence the parameters can be chosen arbitrarily close to zero so that
\eqref{good2} holds at every point of $K$.
Moreover, $B_\lambda\to B$ together with all derivatives on $K$.

We now prove openness. Condition~\eqref{good2} is equivalent to positivity of
\[
G=B_{uu}^2+B_{u\rho}^2+B_{uuu}^2+B_{uu\rho}^2
\]
throughout $K$. By compactness of $K$, $m=\min_K G>0$.
For a nearby solution $\widetilde B$, obtain $\widetilde G$ by replacing
$B$ with $\widetilde B$.
If the derivatives up to order three are uniformly close, then
$\widetilde G$ is uniformly close to $G$, since it is a finite sum of squares.
Thus, for all sufficiently close solutions,
\[
\sup_K|\widetilde G-G|<m/2,\qquad \widetilde G>m/2>0.
\]
Consequently, \eqref{good2} holds for all sufficiently close solutions.

We verify the transversality condition~\eqref{trans2}.
Differentiating \eqref{bell} with respect to $u$ gives
\[
\rho t_{\rho\rho}+2t_\rho=-\alpha(\rho)t_{uu}.
\]
At a point where $t_u=t_\rho=0$, we have
\[
\det\begin{pmatrix}t_{uu}&t_{u\rho}\\t_{\rho u}&t_{\rho\rho}\end{pmatrix}
=-\frac{\alpha(\rho)}{\rho}t_{uu}^2-t_{u\rho}^2<0
\]
by \eqref{good2}. Thus \eqref{trans2} holds.
\end{proof}

\section{Dropping cusp singularities: $\rho_*=0$}

\subsection{}
Consider analytic mappings $f:(u,\rho)\mapsto(x,t)$ defined in an open
neighborhood $U$ of $(u_*,0)$. Equip this space of mappings with the
Whitney topology induced from $C^\infty(U,\mathbb R^2)$.

\begin{samepage}
\begin{lemma}\label{lem3}
Let $L$ be the closed stratified subset of the space of 1-jets defined by
\[
L=\{\rho=0,\ t_u=0\}.
\]
The mappings $f$ for which $j^1f$ is transverse to $L$ form an open dense
subset of the space of analytic mappings. At each $(u,0)\in U$, the
transversality condition is
\begin{equation}\label{trans3}
t_u(u,0)=0\quad\Longrightarrow\quad t_{uu}(u,0)\ne0.
\end{equation}
\end{lemma}
\end{samepage}
\begin{proof}
The set $L$ has a single smooth stratum of codimension two.
Whitney condition~(a) is satisfied. Transversality is equivalent to
\eqref{trans3}, since
\[
\det\begin{pmatrix}0&1\\t_{uu}(u,0)&t_{u\rho}(u,0)\end{pmatrix}
=-t_{uu}(u,0)\ne0.
\]
Openness and density are proved as in Lemma~\ref{lem2}.
\end{proof}

\subsection{}
Let $\alpha(\rho)$ be analytic in a neighborhood of zero, with $\alpha(0)>0$.
Consider both equations
\begin{equation}\label{baxis}
\rho B_{\rho\rho}+2B_\rho=\pm\alpha(\rho)B_{uu}.
\end{equation}
Fix a sufficiently small open neighborhood of $(u_*,0)$.
Its choice is independent of the original solution $B$.
Consider solutions analytic in this neighborhood.
The existence of such solutions was established in \cite{accepted}.
Using the formulas for the derivatives of the hodograph transformation
and \eqref{schwarz}, we obtain
\[
\begin{cases}
x_u=ut_u-\rho t_\rho,\\
x_\rho=ut_\rho\mp\alpha t_u,
\end{cases}
\]
\[
j=\pm\alpha t_u^2-\rho t_\rho^2.
\]
Thus, at the point $(u_*,0)$ under consideration, $j(u_*,0)=0$ is
equivalent to $t_u(u_*,0)=B_{uu}(u_*,0)=0$.

Let $K$ be a sufficiently small compact neighborhood of $(u_*,0)$
in the relative topology of the half-plane $\rho\ge0$, contained in the
chosen open neighborhood. Set $I=K\cap\{\rho=0\}$.
On restrictions of solutions to $K$, we use the Whitney topology,
which here coincides with uniform convergence of all derivatives.

\begin{theorem}\label{th3}
The set of analytic solutions of \eqref{baxis} satisfying the following condition
at every point of $I$ is open and dense:
\begin{equation}\label{good3}
t_u(u,0)=0\quad\Longrightarrow\quad t_{uu}(u,0)\ne0.
\end{equation}
Wherever $t_u(u,0)=0$ on $I$, the transversality condition~\eqref{trans3} holds for $t=B_u$.
\end{theorem}
For $\rho_*=0$, the equality $t_u(u_*,0)=0$ means $b_{20}=0$, while
$t_{uu}(u_*,0)=6b_{30}$. Thus \eqref{good3} is equivalent to $b_{30}\ne0$
at such a point. The relation between this condition and a dropping
cusp $A_3$ was established in \cite{proval,mss24}.
\begin{proof}
To prove density, fix an arbitrary analytic solution $B$ and set
\[
P(u,\rho)=\frac{u^2}{2}+A(\rho),\qquad
A'(\rho)=\pm\frac1{\rho^2}\int_0^\rho s\alpha(s)\,ds.
\]
We check that $P$ is analytic and satisfies \eqref{baxis}. From the convergent series
$\alpha(\rho)=\sum_{n\ge0}\alpha_n\rho^n$, we obtain
\[
\int_0^\rho s\alpha(s)\,ds
=\sum_{n\ge0}\frac{\alpha_n}{n+2}\rho^{n+2},\qquad
A'(\rho)=\pm\sum_{n\ge0}\frac{\alpha_n}{n+2}\rho^n.
\]
Thus the expression for $A'(\rho)$ extends analytically to $\rho=0$,
with $A'(0)=\pm\alpha(0)/2$. Consequently, $A(\rho)$ is analytic.
Differentiating $\rho^2A'=\pm\int_0^\rho s\alpha(s)\,ds$ gives
$\rho A''+2A'=\pm\alpha$; at $\rho=0$, the equality holds by continuity.
Since $P_{uu}=1$, the function $P$ satisfies \eqref{baxis}.

Thus $B_\varepsilon=B+\varepsilon P$ is an analytic solution of the same
equation, with
\[
(B_\varepsilon)_{uu}(u,0)=B_{uu}(u,0)+\varepsilon,\qquad
(B_\varepsilon)_{uuu}(u,0)=B_{uuu}(u,0).
\]
By Sard's theorem, $\varepsilon$ can be chosen arbitrarily small so that
$-\varepsilon$ is a regular value of $u\mapsto B_{uu}(u,0)$.
Then, at every $(u,0)\in I$ where $B_{uu}(u,0)=-\varepsilon$, we have
$B_{uuu}(u,0)\ne0$. Hence all zeros of $(B_\varepsilon)_{uu}(u,0)$ are
simple, and \eqref{good3} holds throughout $I$.
Convergence $B_\varepsilon\to B$ together with all derivatives on $K$
proves density.

We now prove openness. Condition~\eqref{good3} is equivalent to positivity
of the continuous function $B_{uu}(u,0)^2+B_{uuu}(u,0)^2$ on the compact set $I$.
Hence
\[
m=\min_{(u,0)\in I}\bigl(B_{uu}(u,0)^2+B_{uuu}(u,0)^2\bigr)>0.
\]
If the derivatives of $\widetilde B$ up to order three are uniformly close
to those of $B$, then the corresponding sums of squares are uniformly close on $I$.
Thus, for all sufficiently close solutions $\widetilde B$,
\[
\sup_{(u,0)\in I}\left|
\bigl(\widetilde B_{uu}(u,0)^2+\widetilde B_{uuu}(u,0)^2\bigr)
-\bigl(B_{uu}(u,0)^2+B_{uuu}(u,0)^2\bigr)\right|<m/2.
\]
Therefore $\widetilde B_{uu}(u,0)^2+\widetilde B_{uuu}(u,0)^2>m/2$
throughout $I$, and \eqref{good3} holds.
\end{proof}


\bigskip
\begin{thebibliography}{99}
\bibitem{rakh92} A.Kh. Rakhimov.
{\it Singularities of solutions of quasilinear equations} //
St. Petersburg Math. J. {\bf 4}:4, 813--818 (1993).
\bibitem{rakh93} A.Kh. Rakhimov.
{\it Singularities of Riemann invariants} //
Funct. Anal. Appl. {\bf 27}:1, 39--50 (1993).
\bibitem{dub06} B. Dubrovin.
{\it On Hamiltonian Perturbations of Hyperbolic Systems of Conservation Laws,
II: Universality of Critical Behaviour} //
Comm. Math. Phys. {\bf 267}:1, 117--139 (2006).
\bibitem{dgk09} B. Dubrovin, T. Grava, C. Klein.
{\it On Universality of Critical Behavior in the Focusing Nonlinear
Schr\"odinger Equation, Elliptic Umbilic Catastrophe and the Tritronqu\'ee
Solution to the Painlev\'e-I Equation} //
J. Nonlinear Sci. {\bf 19}:1, 57--94 (2009).
\bibitem{ko17} B.G. Konopelchenko, G. Ortenzi.
{\it Jordan form, parabolicity and other features of change of type transition
for hydrodynamic type systems} // J. Phys. A: Math. Theor.
{\bf 50}:21, 215205 (2017).
\bibitem{ko20} B.G. Konopelchenko, G. Ortenzi.
{\it On the plane into plane mappings of hydrodynamic type. Parabolic case} //
Rev. Math. Phys. {\bf 32}:3, 2050006 (2020).
\bibitem{ko13} B.G. Konopelchenko, G. Ortenzi.
{\it Quasi-Classical Approximation in Vortex Filament Dynamics.
Integrable Systems, Gradient Catastrophe, and Flutter} //
Stud. Appl. Math. {\bf 130}:2, 167--199 (2013).
\bibitem{bogaevsky} I.A. Bogaevsky, D.V. Tunitsky.
{\it Singularities of Multivalued Solutions of Quasilinear Hyperbolic Systems} //
Proc. Steklov Inst. Math. {\bf 308}, 67--78 (2020).
\bibitem{kud95} V.R. Kudashev, B.I. Suleimanov.
{\it Characteristic features of some typical spontaneous intensity collapse
processes in unstable media} // JETP Lett. {\bf 62}:4, 382--388 (1995).
\bibitem{kud01} V.R. Kudashev, B.I. Suleimanov.
{\it The effect of small dissipation on the onset of one-dimensional shock waves} //
J. Appl. Math. Mech. {\bf 65}:3, 441--451 (2001).
\bibitem{gar10} R.N. Garifullin, B.I. Suleimanov.
{\it From weak discontinuities to nondissipative shock waves} //
J. Exp. Theor. Phys. {\bf 110}:1, 133--146 (2010).
\bibitem{suldis} B.I. Suleimanov.
{\it Some Generic Singularities of Solutions of Equations with a Small Parameter}.
Doctor of Sciences dissertation, Ufa (2009) (in Russian).
\bibitem{sul12} B.I. Suleimanov.
{\it Generic singularities in solutions of the shallow water equations} //
Dokl. Math. {\bf 85}:1, 125--128 (2012).
\bibitem{chap24} A.M. Shavlukov.
{\it On Generic Singularities of Solutions to the 1D Gas Flow Equations:
Chaplygin and Bechert--Stanyukovich Cases} //
Lobachevskii J. Math. {\bf 45}:6, 2779--2791 (2024).
\bibitem{proval} B.I. Suleimanov, A.M. Shavlukov.
{\it A Typical Dropping Cusp Singularity of Solutions to Equations
of a One-Dimensional Isentropic Gas Flow} //
Bull. Russ. Acad. Sci. Phys. {\bf 84}:5, 552--554 (2020).
\bibitem{ss22} B.I. Suleimanov, A.M. Shavlukov.
{\it Inheritance of Generic Singularities of Solutions of a Linear Wave
Equation by Solutions of Isoentropic Gas Motion Equations} //
Math. Notes. {\bf 112}:4, 608--620 (2022).
\bibitem{ss24} B.I. Suleimanov, A.M. Shavlukov.
{\it Umbilic Singularities of Semiclassical Approximations to Solutions
of the Focusing Nonlinear Schr\"odinger Equation} //
Math. Notes. {\bf 116}:6, 1361--1372 (2024).
\bibitem{mss24} S.N. Melikhov, B.I. Suleimanov, A.M. Shavlukov.
{\it Typical Dropping Asymptotics in the Semiclassical Approximations
to Solutions of the Nonlinear Schr\"odinger Equation} //
Differ. Equ. {\bf 60}:5, 590--602 (2024).
\bibitem{shav} A.M. Shavlukov.
{\it Singularities of Solutions of the One-Dimensional Equations of
Gas Dynamics and Nonlinear Geometrical Optics}.
Candidate of Sciences dissertation, Ufa (2025) (in Russian).
\bibitem{accepted} S.N. Melikhov, B.I. Suleimanov, A.M. Shavlukov.
{\it A Generalization of Korobeinik's Theorem and Generic Dropping Cusp
Singularities of Semiclassical Approximations to Solutions of the
Nonlinear Schr\"odinger Equation} //
Differentsial'nye Uravneniya. {\bf 62}:10, 1383--1398 (2026) (in Russian).
\bibitem{sedykh} V.D. Sedykh.
{\it Mathematical Methods of Catastrophe Theory}.
MCCME, Moscow (2021) (in Russian).
\bibitem{hirsch} M.W. Hirsch. {\it Differential Topology}.
Springer, New York (1976).
\bibitem{trotman} D.J.A. Trotman.
{\it Stability of transversality to a stratification implies Whitney
(a)-regularity} // Invent. Math. {\bf 50}, 273--277 (1979).
\bibitem{ry} B.L. Rozhdestvenskii, N.N. Yanenko.
{\it Systems of Quasilinear Equations and Their Applications to Gas Dynamics}.
Translations of Mathematical Monographs, vol.~55.
American Mathematical Society, Providence, RI (1983).
\bibitem{shvartsburg} A.B. Shvartsburg.
{\it Evolution of a modulated wave pulse in a medium with a saturated
non-linearity} // Sov. Phys. JETP. {\bf 43}:5, 855--860 (1976).
\bibitem{courant} R. Courant, D. Hilbert.
{\it Methods of Mathematical Physics. Vol.~II: Partial Differential Equations}.
Interscience Publishers, New York (1962).
\bibitem{grauert} H. Grauert.
{\it On Levi's problem and the imbedding of real-analytic manifolds} //
Ann. of Math. {\bf 68}:2, 460--472 (1958).
\end{thebibliography}
\end{document}